\documentclass[12pt,a4paper]{amsart}
\usepackage[top=35mm, bottom=35mm, left=30mm, right=30mm]{geometry}

\usepackage{xcolor}
\usepackage{amsfonts} 
\usepackage[looser]{newtxtext}
\usepackage{newtxmath}
\usepackage{verbatim}

\usepackage[hidelinks,pagebackref]{hyperref}
\usepackage[nobysame,msc-links]{amsrefs}

\usepackage{fancyhdr}
\makeatletter
\def\@@and{\MakeLowercase{and}}
\makeatother

\newcommand{\eps}{\varepsilon}
\newcommand{\bbr}{\mathbb{R}}
\newcommand{\bbn}{\mathbb{N}}
\newcommand{\bbq}{\mathbb{Q}}
\newcommand{\calv}{\mathcal{V}}
\newcommand{\calm}{\mathcal{M}}
\newcommand{\calmx}{\mathcal{M}(X)}
\newcommand{\dd}{\mathop{}\!\mathrm{d}}

\newtheorem{thm}{Theorem}[section]

\newtheorem{lem}[thm]{Lemma}
\newtheorem{prop}[thm]{Proposition}
\newtheorem{cor}[thm]{Corollary}

\theoremstyle{definition}
\newtheorem{remark}[thm]{Remark}

\numberwithin{equation}{section}

\allowdisplaybreaks[4]

\DeclareMathOperator{\per}{Per}
\DeclareMathOperator{\conv}{conv}
\DeclareMathOperator{\supp}{supp}

\title[T\MakeLowercase{he mean orbital pseudo-metric and the space of invariant measures}] %Use the shortened version of the full title
{T\MakeLowercase{he mean orbital pseudo-metric and the space of invariant measures}}

\author[J. L\MakeLowercase{i}]{J\MakeLowercase{ian} Li}
\address[J. Li]{Department of Mathematics,
	Shantou University, Shantou, 515063, Guangdong, China}
\email{lijian09@mail.ustc.edu.cn}
\urladdr{https://orcid.org/0000-0002-8724-3050}

\author[Y. X\MakeLowercase{iao}]{Y\MakeLowercase{uanfen} Xiao}
	\address[Y. Xiao]{School of Mathematical Sciences, Xiamen University, Xiamen, 
		361005, Fujian, China}
	\email{xyuanfen@mail.ustc.edu.cn}
\urladdr{https://orcid.org/0000-0002-8468-271X}

	
 
\begin{document}

\maketitle

\begin{abstract}
We study the mean orbital pseudo-metric for Polish dynamical systems and its connections with properties of the space of invariant measures. 
We give equivalent conditions for when the set of invariant measures generated by periodic points is dense in the set of 
ergodic measures and the space of invariant measures. 
We also introduce the concept of asymptotic orbital average shadowing property and show that it implies that every non-empty compact connected  subset of the space of invariant measures has a generic point.
\end{abstract}

\section{Introduction}

By a Polish dynamical system, we mean a pair $(X,T)$ where $X$ is a Polish space and $T\colon X\to X$ is a continuous map.
A Borel probability measure $\mu$ on $X$ is called $T$-invariant 
if $\mu(A)=\mu(T^{-1}A)$ for all Borel subset of $X$. By the well-known Kryloff-Bogoliouboff theorem, there always exists some invariant measure provided that $X$ is compact. 
If $(X,T)$ has some periodic points, then any periodic orbit supports a unique invariant measure.
Endow the space $\calm_T(X)$ of all $T$-invariant measures with the weak$^*$-topology.  
This paper is devoted to  exploring the space of invariant measures.
Especially, we will consider the following natural question: 
when the set of invariant measures generated by periodic points
is dense in the space $\calm_T(X)$.

In~\cite{Par1961} Parthasarathy obtained that if $X$ is a product of countably many copies of a Polish space and $T$ is the shift transformation,  then the set of invariant measures generated by periodic points is dense in $\calm_T(X)$. 
In~\cite{Sig1974} Sigmund proved that if a compact dynamical system $(X,T)$ has the periodic specification property then the set of invariant measures generated by periodic points is dense in $\calm_T(X)$. 
Recently in~\cite{GK2018} Gelfert and Kwietniak introduced two properties 
(linkability and closablity) related to the set of periodic points 
and showed that if  a Polish dynamical system is closable with respect to a linkable subset of periodic points then the set of invariant measures generated by periodic points is dense in $\calm_T(X)$. 

Let $(X,T)$ be a Polish dynamical system and  $d$ be a bounded compatible complete metric on $X$. 
The Besicovitch pseudo-metric on $X$ which measures the average separation of points along orbits of a dynamical system $(X,T)$ is defined by
\[  
D_B(x,y) = \limsup_{n\to\infty} \frac{1}{n} \sum_{i=0}^{n-1} d(T^ix,T^iy)
\]
for $x,y\in X$.
In \cite{KLO2017}, Kwietniak et al.\@  studied the Besicovitch pseudo-metric for a compact dynamical system and showed that the map from a point $x$ to the set $\calv_T(x)$ of invariant measures quasi-generated by $x$( that is, $\calv_T(x)=\{\mu\in \calm_T(X)\colon \text{there exists a sequence } \{k_n\}_{n=1}^{\infty} \newline 
\text{such that } \lim_{n\to\infty}\delta_{T^{k_n}x}=\mu  \} $) is uniformly continuous from $X$ equipped with the  Besicovitch pseudo-metric to the space of non-empty compact subsets of $\calm_T(X)$ equipped with the Hausdorff metric.
Very recently, in \cite{ZZ2020},  L. Zheng and Z. Zheng introduced a weak form of Besicovitch pseudo-metric which ignores the order of points on the orbits.
Given a dynamical system $(X,T)$, the mean orbital pseudo-metric 
$
\bar E$ on $X$ is defined by 
\[  
\bar E(x,y) = \limsup_{n\to\infty} \min_{\sigma\in S_n} 
\frac{1}{n} \sum_{i=0}^{n-1} d(T^ix,T^{\sigma(i)}y)
\]
for $x,y\in X$,
where $S_n$ denotes the permutation group on $\{0,1,\dotsc,n-1\}$.
It should be noticed that the authors of \cite{ZZ2020} did not give $\bar E$ any name  and denoted it by $\bar F$.
The authors of \cite{CKLP2022} proposed the name of mean orbital pseudo-metric.
They also  decided to change the notation to
avoid confusion with the Feldman-Katok pseudo-metric,  which is denoted by $\bar F_K$.
L. Zheng and Z. Zheng proved in \cite{ZZ2020} that a compact  dynamical system $(X,T)$ is uniquely ergodic if and only if $\bar E(x,y)=0$ for all $x,y\in X$.
Cai et al. showed in \cite{CKLP2022}  that for a compact dynamical system $(X,T)$ the map $x\mapsto \calv_T(x)$ is uniformly continuous from $X$ equipped with the mean orbital pseudo-metric to the space of non-empty compact subsets of $\calm_T(X)$ equipped with the Hausdorff metric. 
It is observed in  \cite{CKLP2022}  and \cite{XZ2021} that the mean orbital pseudo-metric is closely related to the Wasserstein metric on the space of Borel probability measures. 

The aim of this paper is to apply another version of mean orbital pseudo-metric to study the space of invariant measures. 
Inspirited by the properties of linkability and closablity, we obtain the following equivalent
conditions for when the set of invariant measures generated by periodic points is dense
in the set of ergodic measures or the space of invariant measures, 
which is one of the main results of this paper.

\begin{thm} \label{thm:main-result1}
Let $(X,T)$ be a Polish dynamical system and $K\subset \per(T)$ (the set of periodic points).
Then 

\begin{enumerate} 
\item the set of invariant measures generated by points in $K$ is 
dense in the set of ergodic measures of $(X,T)$ if and only if 
 for every regular point $y\in X$ (generic for some ergodic measure), $\eps>0$ and $N\in\bbn$ there exist $x\in K$
and $n\in\bbn$ such that $n\geq N$, $T^nx =x$ and 
\[
\min_{\sigma\in S_{n} }\frac{1}{n}\sum_{i=0}^{n-1}d(T^ix,T^{\sigma(i)}y )<\eps;
\]
\item the set of invariant measures generated by points in $K$ is 
dense in the set of invariant measures of $(X,T)$ if and only if 
for every $k\geq 1$ and regular points $y_1,y_2,\dotsc,y_k\in X$, $\eps>0$ and $N\in\bbn$
there exist $x\in K$ and $n\in\bbn$ such that $n\geq N$, $T^{kn}x=x$ and   
\[
\min_{\sigma\in S_{kn}}\frac{1}{kn}
\sum_{i=0}^{kn-1}d(z_i,T^{\sigma(i)}x)<\eps, 
\] 
where 
$ z_{i+(j-1)n}=T^iy_j $ for $ 0\leq i\leq n-1 $ and $ 1\leq j\leq k $.
\end{enumerate}
\end{thm} 

As applications, we show that for a continuous interval map $f\colon [0,1]\to [0,1]$ 
if it has zero topological entropy then the set of invariant measures generated by periodic points of $f$ is
dense in the set of ergodic measures of $f$, and if it is transitive then the set of invariant measures generated by periodic points of $f$ is 
dense in the set of invariant measures of $f$.

Another important consequence of a compact dynamical system $(X,T)$ with the 
specification property is that for every non-empty closed connected subset $V$ of $\calm_T(X)$ there exists a point $x\in X$ such that $\calv_T(x)=V$ 
\cites{Sig1974,Sig1977}.
This result was generalized for compact dynamical systems with various types of weak specification property.
In particular, it is shown in \cite{DTY2015} and \cite{KLO2017}
if a compact dynamical system has the asymptotic average shadowing property
then for every non-empty closed connected subset $V$ of $\calm_T(X)$ there exists a point $x\in X$ such that $\calv_T(x)=V$.
In this paper we introduce a weak version of asymptotic average shadowing property.
We say that a Polish dynamical system $(X,T)$  has the asymptotic orbital average shadowing property if  for every asymptotic average pseudo-orbit  $\{x_i\}_{i=0}^\infty$ of $X$, that is,
\[
\lim_{n\to\infty}\frac{1}{n}\sum_{i=0}^{n-1}d(Tx_i,x_{i+1})=0,
\]
there exists $x\in X$ such that
\[
\lim_{n\to\infty} \min_{\sigma\in S_n }\frac{1}{n}\sum_{i=0}^{n-1}d(x_i,T^{\sigma(i)}x)=0.
\]

The second main result of this paper is as follows.

\begin{thm}\label{thm:AOASP-V}
Let $(X,T)$ be a Polish dynamical system and  $\calv_T(x)=\{\mu\in \calm_T(X)\colon \text{ there } \newline 
\text{exists a sequence } \{k_n\}_{n=1}^{\infty}   
\text{such that } \lim_{n\to\infty}\delta_{T^{k_n}x}=\mu  \} $.
If $(X,T)$ has the asymptotic orbital average shadowing property,
then for every non-empty compact connected  subset $V$ of $\calm_T(X)$, there exists $x\in X$ with $\calv_T(x)=V$. 
\end{thm}

The paper is organized as follows. Section 2 contains preliminaries on the space of Borel probability measures on a Polish space and topological dynamics. 
Section 3 is devoted to studying the question when the set of invariant measures generated by periodic points is dense in the space  of invariant measures and proving Theorem 1.1. 
Theorem 1.2 is proved in Section 4.

\section{Preliminaries}

\subsection{Polish space and Borel probability measure}
Let $X$ be a  Polish space, that is, $X$ is a topological space, which is separable and
completely metrizable. Fix a compatible bounded complete metric $d$ on $X$.
Let $C_b(X)$ be the collection of all bounded real functions on $X$.
We say that  $f\colon X\to\bbr$ is a Lipschitz function 
if there exists a constant $L\geq 0$ such that $|f(x)-f(y)|\leq Ld(x,y)$ for all $x,y\in X$.
Let 
\[
p_L(f)=\sup_{x\neq y} \frac{|f(x)-f(y)|}{d(x,y)}.
\]
Then it is clear that $f$ is a Lipschitz function if and only if $p_L(f)<\infty$.

Let $\calmx$ be the set of all Borel probability measures on $X$. 
For $\mu\in\calmx$, the support of $\mu$, denoted by $\supp(\mu)$, is the smallest closed subset of $X$ with full $\mu$-measure.
For $x\in X$, let $\delta_x\in\calmx$ be the Dirac measure supported on $\{x\}$.
The weak$^*$-topology on $\calmx$ is the weakest topology making each of the
evaluation maps $\mu\mapsto \int f\dd \mu$ continuous for any $f\in C_b(X)$.

For $\mu,\nu\in\calmx$, let 
\[
\gamma(\mu,\nu)=
\sup\biggl\{\biggl|\int_X f\dd \mu-\int_X f\dd \nu\biggr|\colon f\in C_b(X) \text{ with } p_L(f)  \leq 1\biggr\}.
\]
We say that a measure $\pi\in\calm(X\times X)$
is a coupling of $\mu$ and $\nu$ if $\pi(A\times X)=\mu(A)$ and $\pi(X\times B)=\nu(B)$ for all Borel subsets $A,B$ of $X$.
Denote by $\Pi(\mu,\nu)$ the collection of all couplings of $\mu$ and $\nu$.
The Wasserstein metric, or Kantorovich-Rubinshtein metric, on $\calmx$ is defined by
\[
W_1(\mu,\nu)=\inf_{\pi\in \Pi(\mu,\nu)}\int_{X\times X}d(x,y)\dd \pi(x,y)
\]
for  $\mu,\nu\in\calmx$.
By the well-known Kantorovich–Rubinshtein theorem, we have 
$W_1(\mu,\nu)=\gamma(\mu,\nu)$ for all $\mu,\nu \in\calmx$ (see e.g.\ \cite{Gar2018}*{Corollary 21.2.3} or \cite{Dud2002}*{Theorem 11.8.2}).
By the boundness of $d$, the metrics $W_1$ and $\gamma$ are complete and compatible with the 
weak$^*$-topology on $\calmx$ (see e.g.\ \cite{Gar2018}*{Corollary 21.2.4}). 
Even though the metrics $\gamma$ and $W_1$ are equal,  the metric $\gamma$ is more easy to handle for our purpose in this paper. So we will use the metric $\gamma$ instead of $W_1$ in the rest of this article.

The following result must be folklore, e.g. see the example given in \cite{Gar2018}*{pp. 303--304}. 
A slight weak version appears in \cite{CKLP2022}*{Lemma 3.3} and \cite{XZ2021}*{Theorem A.1}.
Since the proof is almost the same, we leave it to the reader.
\begin{lem} \label{lem:finite-seq-metric}
Let $X$ be a Polish space and $d$ be a compatible bounded complete metric on $X$. Then
for any $n\in\mathbb{N}$, $x_0,x_1,\dotsc,x_{n-1}$, $y_0,y_1,\dotsc,y_{n-1}\in X$,
\[  
\min_{\sigma\in S_n}
\frac{1}{n}\sum_{i=0}^{n-1}  d(x_i,y_{\sigma(i)})
=\gamma\biggl(\frac{1}{n}\sum_{i=0}^{n-1}  \delta_{x_i},
\frac{1}{n}\sum_{i=0}^{n-1} \delta_{y_i}\biggr),
\]
where $ S_n $ is the permutation group of $ \{0,1,\dotsb,n-1\} $.
\end{lem}

\subsection{Topological dynamics}
Let $X$ be a Polish space and $T\colon X\to X$ be a continuous map.
The pair $ (X,T) $ is called a Polish dynamical system. 

For a point $ x\in X$, the orbit of $x$ is the set $\{T^nx\colon n\geq 0\}$ 
and the $\omega$-limt set of $x$, denoted by $\omega_f(x)$, is the collection of all the limit points of subsequences of $\{T^nx\}_{n=0}^\infty$.
We say that $(X,T)$ is transitive if there exists some $x\in X$ with $\omega_f(x)=X$,
and weakly mixing if the product system $(X\times X,T\times T)$ is transitive.

A Borel probability measure $\mu\in\calmx$ is called $T$-invariant if 
$\mu(T^{-1}(B))=\mu(B)$ for all Borel subset $B$ of $X$.
Denote by $\calm_T(X)$ the set of all the $T$-invariant Borel probability measures. 
It is clear that $\calm_T(X)$ is a closed convex subset of $\calmx$.
If $X$ is compact, by the Kryloff-Bogoliouboff theorem $\calm_T(X)$ is not empty, see e.g.\@ \cite{DGS1976}.
But if $X$ is only Polish, $\calm_T(X)$ can be empty. A simple example is that $T\colon \bbr\to\bbr$, $x\mapsto x+1$.
It is shown in \cite{Oxt1952}*{(7.2)} that for a Polish dynamical system $(X,T)$, $\calm_T(X)$ is not empty if and only if there exists a point $x\in X$ and a compact subset $A$ of $X$ such that $\{n\in\bbn\colon T^nx\in A\}$ has positive upper density.

For any $x\in X$ and $n\in\bbn$, the $n$th-empirical measure of $x$ is defined as \[
m_T(x,n)=\frac{1}{n}\sum_{i=0}^{n-1}\delta_{T^ix}.
\] 
According to  Lemma~\ref{lem:finite-seq-metric}, for every $n\in\bbn$ and $x,y\in X$, one has 
\[
\min_{\sigma\in S_n}\frac{1}{n} \sum_{i=0}^{n-1}d(T^ix,T^{\sigma(i)}y)
=\gamma(m_T(x,n),m_T(y,n)).
\]
If $\mu\in\calmx$ is a limit of some subsequence of $\{m_T(x,n)\}_{n=1}^\infty$,
then $\mu$ is $T$-invariant.
In particular, if the sequence $\{m_T(x,n)\}_{n=1}^\infty$ converges to $\mu$,
then we say that $x$ is a generic point for $\mu$. 

A point $x\in X$ is called quasi-regular if it is generic for some invariant measure. In this case, denote by $\mu_x$ the unique measure for which $x$ is generic. Let $Q_T(X)$ be the collection of all quasi-regular points in $X$.

A $T$-invariant measure $\mu\in\calm_T(X)$ is called ergodic 
if for every Borel subset $B$ of $X$, $T^{-1}B= B$ implies that $\mu(B)=0$ or $1$.
Denote by  $\calm_T^e(X) $  the set of all the ergodic $T$-invariant measures.
A point $x\in X$ is called regular if it is generic for some ergodic measure.
By the Birkhoff ergodic theorem, every ergodic measure has some generic points.
The ergodic decomposition also holds for Polish dynamical system (see e.g. \cite{Oxt1952}*{Section 8}.
Therefore, $\calm_T(X)$ is not empty if and only if so is $\calm_T^e(X)$ if and only if so is $Q_T(X)$.

We say that a point $ x\in X $ is a periodic point with period $ p $, if the positive integer $ p $ is the least positive integer such that $ T^px=x $. 
Denote the set of all the periodic points by $ \per(T) $. 
If $x$ is a periodic point with period $p$,
then $m_T(x,p)$ is an ergodic $T$-invariant measure and $m_T(x,p)=m_T(x,kp)$ for all $k\in\bbn$.

\section{The mean orbital pseudo-metric}

Let $(X,T)$ be a Polish dynamical system.
We fix  a compatible bounded complete metric $d$ on $X$.
The mean orbital pseudo-metric $ \overline{E} $ is defined by
\[
\overline{E}(x,y)=\limsup_{n\to\infty}\min_{\sigma\in S_n }\frac{1}{n}\sum_{i=0}^{n-1}d(T^ix,T^{\sigma(i)}y )
\]
with $ x $ and $ y\in X $.
The following lemma reveals the relationship between the mean orbital psedudo-metric $\overline{E}$ and the metric $\gamma$ defined on the measure space.
\begin{lem}\label{lem:MOPE-gamma}
Let $(X,T)$ be a Polish dynamical system.
Then for every $x,y\in X$,
\[
\overline{E}(x,y)=\limsup_{n\to\infty} \gamma(m_T(x,n),m_T(y,n)).
\]
Furthermore, if $x,y$ are quasi-regular points, then $\overline{E}(x,y)=\gamma(\mu_x,\mu_y)$.
\end{lem}
\begin{proof}
The first result follows from Lemma~\ref{lem:finite-seq-metric} directly.
If $x,y$ are quasi-regular points, then $\lim_{n\to\infty} m_T(x,n)=\mu_x$ and $\lim_{n\to\infty} m_T(y,n)=\mu_y$. So the second result follows from the continuity of metric $\gamma$.
\end{proof}

\begin{prop}
Let $(X,T)$ be a Polish dynamical system. Then the collection $Q_T(X)$ of quasi-regular points is closed with respect to the mean orbital pseudo-metric $ \overline{E} $.
\end{prop}
\begin{proof}
Let $\{x_k\}$ be a sequence in $Q_T(X)$ and $x\in X$ with $\lim_{k\to\infty} \overline{E}(x_k,x)=0$.
Fix $\eps$ to be an arbitrary, positive real number.
There exists $k\in\bbn$ such that 
$\overline{E}(x_k,x)<\frac{\eps}{3}$.
By Lemma~\ref{lem:MOPE-gamma}, $\limsup_{n\to\infty} \gamma(m_T(x_k,n),m_T(x,n))<\frac{\eps}{3}$.
There exists $N_1\in\bbn$ such that for any $n>N_1$,
$\gamma(m_T(x_k,n),m_T(x,n))<\frac{\eps}{3}$.
As $x_k$ is a quasi-regular point, it implies that $\lim_{n\to\infty} m_T(x_k,n)=\mu_{x_k}$.
There exists $N_2\in\bbn$ such that for any $i,j>N_2$,
$\gamma(m_T(x_k,i),m_T(x_k,j))<\frac{\eps}{3}$.
Then for any $i,j>\max\{N_1,N_2\}$, 
\begin{align*}
\gamma(m_T(x,i),m_T(x,j))&\leq  \gamma(m_T(x,i),m_T(x_k,i))+
\gamma(m_T(x_k,i),m_T(x_k,j))\\
&\qquad\qquad  +\gamma(m_T(x_k,j),m_T(x,j))\\ &<\frac{\eps}{3}+
\frac{\eps}{3}+\frac{\eps}{3}=\eps.
\end{align*}
This shows that $\{m_T(x,n)\}_{n=1}^\infty$ is a Cauchy sequence.
As $(\calmx,\gamma)$ is complete, $\{m_T(x,n)\}_{n=1}^\infty$ is convergent, and therefore 
$x$ is a quasi-regular point.
\end{proof}

\begin{remark}
Let $(X,T)$ be a Polish dynamical system and $ A $, $ B$ be two subsets of the collection $ Q_T(X) $ of quasi-regular points . 
By Lemma~\ref{lem:MOPE-gamma}, we have that $ \{\mu_x\colon x\in A \} $ is dense in $ \{\mu_y\colon y\in B \} $  in the space $\calm_T(X)$ if and only if  $ A $ is dense in $ B $ with respect to the mean orbital pseudo-metric $ \overline{E} $, that is, for every $y\in B$ and $\eps>0$ there exists $x\in A$ with $\overline{E}(x,y)<\eps$.

If $A$ consists of periodic points, then we have the following result which is easy to handle in some sense.
\end{remark}

\begin{proof}[Proof of Theorem~\ref{thm:main-result1}]
 We divide the proof of Theorem~\ref{thm:main-result1} into two parts. The  part (1) of Theorem~\ref{thm:main-result1} follows from  Theorem~\ref{thm:per-dense-B}. And the part (2) of Theorem~\ref{thm:main-result1} can be found in Theorem~\ref{thm:per-dense-conv-B}.
\end{proof}

\begin{thm}\label{thm:per-dense-B}
Let $ (X,T) $ be a Polish dynamical system,
$A\subset \per(T)$ and $B\subset Q_T(X) $.
Then $ \{\mu_x\colon x\in A \} $ is dense in $ \{\mu_y\colon y\in B \} $  in the space $\calm_T(X)$ if and only if
for every $y\in B$, $\eps>0$ and $N\in\bbn$ there exist $x\in A$
and $n\in\bbn$ such that $n\geq N$, $T^n x= x$ and 
\[
\min_{\sigma\in S_{n} }\frac{1}{n}\sum_{i=0}^{n-1}d(T^ix,T^{\sigma(i)}y )<\eps.
\]
\end{thm}

\begin{proof}
We first prove the necessity. 
Fix $y\in B$, $\eps>0$ and $N\in\bbn$.
There exists $x\in A$ such that $\gamma(\mu_x,\mu_y)<\frac{\eps}{2}$.
Denote by $p$ the period of $x$. Then $\mu_x=m_T(x,p\cdot i)$ for all $i\in\bbn$.
As $y$ is quasi-regular, $\lim_{n\to\infty} m_T(y,n)=\mu_y$.
Pick $n\in\bbn$ such that $n\geq N$ and $p|n$ and $\gamma(m_T(y,n),\mu_y)<\frac{\eps}{2}$.
By Lemma~\ref{lem:finite-seq-metric}, one has 
\begin{align*}
\min_{\sigma\in S_{n} }\frac{1}{k}\sum_{i=0}^{n-1}d(T^ix,T^{\sigma(i)}y )
&= \gamma(m_T(x,n),m_T(y,n)) \\
&< \gamma(\mu_x,\mu_y)+\gamma(\mu_y,m_T(y,n))\\
&<\frac{\eps}{2}+\frac{\eps}{2}=\eps.
\end{align*}

Now we prove the sufficiency.
Fix $y\in B$ and $\eps>0$.
As $y$ is quasi-regular,
there exists $N\in\bbn$ such that $\gamma(m_T(y,n),\mu_y)<\frac{\eps}{2}$ for all $n\geq N$.
By the assumption, for this $N$
there exists $x\in A$ and $n\in\bbn$ such that $n>N$, $T^n x =x$ and 
\[
\min_{\sigma\in S_{n} }\frac{1}{n}\sum_{i=0}^{n-1}d(T^ix,T^{\sigma(i)}y )<
\frac{\eps}{2}.
\]
By Lemma~\ref{lem:finite-seq-metric}, one has 
\begin{align*}
\gamma(\mu_x,\mu_y)&\leq 
\gamma(m_T(x,n),m_T(y,n))+\gamma(m_T(y,n),\mu_y)\\
&< 
\min_{\sigma\in S_{n} }\frac{1}{n}\sum_{i=0}^{n-1}d(T^ix,T^{\sigma(i)}y )
+\frac{\eps}{2} 
\\
&<\frac{\eps}{2}+\frac{\eps}{2}=\eps.
\end{align*}
This ends the proof.
\end{proof}

In \cite{GK2018}, Kelfert and  Kwietniak introduced the notion of closability.
Let $ (X,T) $ be a Polish dynamical system and $K\subset \per(T)$.
A point $ x\in X $ is called $ K $-closable 
if for any $ \eps >0 $ and $ N\in\bbn $, there exist two  positive integers $p=p(x,\eps,N)  $ and $ q=q(x,\varepsilon,N) $ with $ N\leq p\leq q\leq (1+\eps)p $ such that there is some $ y\in K $ satisfying $ T^qy=y $ and $ d(T^iy,T^ix)<\varepsilon $ for $ 0\leq i\leq p-1 $. 
We say that $ (X,T) $ is $ K $-closable if for every $ \mu\in \mathcal{M}_T^e(X) $, there is some generic point for $ \mu $ that is $ K $-closable.  

\begin{lem}\label{close-gamma}
Let $ (X,T) $ be a Polish dynamical system and $K\subset \per(T)$.
If a point $x\in X$ is $K$-closable, 
then for every $\eps>0$ and $N\in\bbn$,
	there exists a periodic point $y\in K$ with period $q$ such that $q\geq N$ and  
\[
\frac{1}{q}\sum_{i=0}^{q-1}d(T^ix,T^{i}y )<\eps.
\]
\end{lem}

\begin{proof}
As $d$ is a bounded metric on $X$, without loss of generality assume that the 
diameter of $X$ is $1$.
Fix $\eps>0$ and $N\in\bbn$.
By the assumption that $x$ is $K$-closable,
there exist two  positive integers $p $ and $ q $ with $ N\leq p\leq q\leq (1+\eps/2)p $ such that there is some $ y\in K $ satisfying $ T^qy=y $ and $ d(T^ix,T^iy)<\eps/2 $ for $ 0\leq i\leq p-1 $. 
Then 
\begin{align*}
\frac{1}{q}\sum_{i=0}^{q-1}d(T^ix,T^{i}y ) 
&\leq \frac{1}{q}\Bigl(\sum_{i=0}^{p-1}d(T^ix,T^iy) +q-p\Bigr) \\
& <\frac{1}{q}\Bigl( \frac{\eps}{2} p +q-p\Bigr)\\
& \leq\frac{1}{q}\Bigl( \frac{\eps}{2} p +\frac{\eps}{2} p\Bigr)\leq \eps. \qedhere 
\end{align*}
\end{proof}

Combining Theorem~\ref{thm:per-dense-B} and Lemma~\ref{close-gamma},
we have the following consequence.

\begin{cor}[\cite{GK2018}*{Theorem 4.11}]
Let $ (X,T) $ be a Polish dynamical system and $K$ be a subset of $\per(T)$.
If $(X,T)$ is $K$-closable,
then $\{\mu_x\colon x\in K\}$ is dense in $\calm^e_T(X)$.
\end{cor}

It is shown in \cite{GK2018}*{Propositions 4.8, 4.9 and 4.10, respectively} that every dynamical system 
with the periodic specification property ($\beta$-shift, $S$-gap shift, respectively) is closable 
with respect to the set of periodic points.  
Here we examine continuous interval maps with zero entropy.
We refer the reader to \cite{BC1992} and \cite{Rue2017} for preliminaries of dynamics on interval maps. 

\begin{prop}
Let $f\colon [0,1]\to[0,1]$ be a continuous map.
\begin{enumerate}
    \item  If $f$ has zero topological entropy then $\{\mu_x\colon x\in \per(f)\}$ is dense in $\calm^e_f([0,1])$.
    \item If $f$ is not Li-Yorke chaotic then it is $\per(f)$-closable.
\end{enumerate}

\end{prop}
\begin{proof}
Let $\mu$ be an ergodic invariant measure of $f$.
Denote the support of $\mu$ by  $\supp(\mu)$. 
Then $(\supp(\mu),f)$ is transitive. 
In particular, either $\supp(\mu)$ is a periodic orbit or  $\supp(\mu)$ is perfect.
If $\supp(\mu)$ is a periodic orbit of $z$,
then $z$ is generic for $\mu$ and $z$ is $\per(f)$-closable clearly.
Now assume that $\supp(\mu)$ is perfect. 
Pick a generic point $y\in \supp(\mu)$ for $\mu$.
Then $\omega_f(y)=\supp(\mu)$.
For every $n\geq 0$ and $i=0,1,\dotsc,2^n-1$, let 
\[
I_n^i = [\min \omega_{f^{2^n}}(f^i(y)), \max \omega_{f^{2^n}}(f^i(y)) ].
\]
Then 
$\omega_f(y)\subset \bigcup_{j=0}^{2^n-1}I_n^j$ for all $n\geq 0$.
Moreover, $I_n^{i+1\pmod {2^n}}\subset f(I_n^i)$ for all $n\geq 0$ and $i=0,1,\dotsc,2^n-1$.
By \cite{BC1992}*{Lemma I.4} or \cite{Rue2017}*{Lemma 1.13}, for each $n\geq 0$ there exists $x_n\in [0,1]\cap \per(f)$ such that $f^i(x_n)\in I_n^i$ 
for $i=0,1,\dotsc,2^n-1$ and $f^{2^n}(x_n)=x_n$.

(1) As $f$ has zero topological entropy, by \cite{BC1992}*{Lemma VI.14} or \cite{Rue2017}*{Proposition 5.24},     
the collection $\{ I_n^i\}_{i=0}^\infty $ of closed intervals are pairwise disjoint for every $n\geq 0$, then there are at most $n$ intervals with diameter larger than $\frac{1}{n}$.
Without loss of generality, assume that $y\in I_n^0$. Then $f^j(y)\in I_n^{j\pmod {2^n}}$ for all $j\geq 0$. 
For every $k\geq 1$, we have 
\begin{align*}
    \frac{1}{k 2^n} \sum_{i=0}^{k 2^n-1}|f^i(y)-f^i(x_n)|
    \leq  \frac{1}{2^n}\bigl(n+(2^n-n)\tfrac{1}{n}\bigr)\to 0
\end{align*}
as $n\to\infty$.
It implies that for any $y\in \supp(\mu) $ being the generic point for $\mu$ and any $\varepsilon>0$, there exist $2^n\in\mathbb{N}$ and $x=x_n\in\per(f)$ such that 
$f^{2^n}(x)=x $ and 
$\frac{1}{ 2^n} \sum_{i=0}^{ 2^n-1}|f^i(y)-f^i(x_n)|<\varepsilon$.
Now by Theorem~\ref{thm:per-dense-B} $\{\mu_x\colon x\in \per(f)\}$ is dense in $\calm^e_f([0,1])$.
 
(2) As $f$ is not Li-Yorke chaotic, by \cite{BC1992}*{Proposition  VIII.34} or \cite{Rue2017}*{Theorem 5.17} $f$ has zero topological entropy.
Furthermore, by \cite{Rue2017}*{Proposition 5.32},
\[
\lim_{n\to\infty} \max_{0\leq i\leq 2^n-1} 
(\max \omega_{f^{2^n}}(f^i(x))-\min \omega_{f^{2^n}}(f^i(x)))=0.
\]
Again assume that $x\in I_n^0$, then for every $i\geq 0$, 
\[
|f^i(x)-f^i(y_n)|\leq \max_{0\leq i\leq 2^n-1} 
(\max \omega_{f^{2^n}}(f^i(x))-\min \omega_{f^{2^n}}(f^i(x)))
\to 0
\]
as $n\to\infty$.
Hence $x$ is $\per(f)$-closable.
\end{proof}

It is clear that $\calmx$ is a convex set.
For a subset $A$ of $\calmx$, the convex hull of $A$, denoted by $\conv(A)$, is  the smallest convex subset of $\calmx$ containing $A$.
The closure of the convex hull of $A$ is denoted by $\overline{\conv}(A)$.

\begin{remark}\label{rem:conv-hull}
Let $ (X,T) $ be a Polish dynamical system 
and $ A $, $ B$ be two subsets of $ Q_T(X) $.
It is clear that 
\[
\conv_\bbq(\{\mu_y\colon y\in B\}):=\Bigl\{ \sum_{j=1}^{p}\lambda_j\mu_{y_j}\colon y_{j}\in B,        \sum_{j=1}^{p}\lambda_j=1, \lambda_i\geq 0,  \lambda_i\in\bbq,  \text{ and } p\geq 1  \Bigr\} 
\]
is dense in $\overline{\conv}(\{\mu_y\colon y\in B\})$.
To show that  $ \{\mu_x\colon x\in A \}$  is dense in  $ \overline{\conv}
(\{\mu_y\colon y\in B \}) $, it is sufficient to show that $ \{\mu_x\colon x\in A \}$  is dense in $\conv_\bbq(\{\mu_y\colon y\in B\})$.
Moreover, if $A$ is a subset of $B$, then it is easy to check that $ \{\mu_x\colon x\in A \}$  is dense in  $ \overline{\conv}
(\{\mu_y\colon y\in B \}) $ if and only if $ \{\mu_x\colon x\in A \}$  is dense in  
\[
\Bigl\{ \alpha \mu_{y_1} +(1-\alpha)\mu_{y_2} \colon y_1,y_2\in B\text{ and } \alpha\in [0,1]\cap \bbq \Bigr\}.
\]
\end{remark}

The second main result in this section is as follows, which implies
the part (2) of Theorem~\ref{thm:main-result1}.

\begin{thm}\label{thm:per-dense-conv-B}
Let $ (X,T) $ be a Polish dynamical system,
$A\subset \per(T)$ and $ B\subset  Q_T(X) $.
Then  $ \{\mu_x\colon x\in A \}$  is dense in  $ \overline{\conv}
(\{\mu_y\colon y\in B \}) $
if and only if 
for any $\eps>0$, any $ k\geq 1 $,  $y_1,y_2,\dotsc,y_k \in  B$ and $N\in\bbn$
there exists $x\in A$  and $n\in\bbn$ such that $n\geq N$, $T^{kn} x= x$ and   
\[
\min_{\sigma\in S_{kn}}\frac{1}{kn}
\sum_{i=0}^{kn-1}d(z_i,T^{\sigma(i)}x)<\eps,
\] 
where 
$ z_{i+(j-1)n}=T^iy_j $ for $ 0\leq i\leq n-1 $ and $ 1\leq j\leq k $.
\end{thm}

\begin{proof}
For the sufficiency,
we only need to prove that $ \{\mu_x\colon x\in A \} $ is dense in $\conv_\bbq(\{\mu_y  \colon \allowbreak y\in B\})$.
Fix $\eps>0$, $ p\geq 1 $, $y_1, y_2,\dotsc, y_p\in B $, $\lambda_1,\lambda_2,\dotsc,\lambda_p$ with $\sum_{j=1}^{p}\lambda_j=1, \lambda_i > 0 $ and $\lambda_i\in\bbq$ for $i=1,2,\dotsc,p$.
Then there  exists $ k \in\bbn$ such that $\lambda_i=\frac{l_i}{k}$ with $l_i> 0$ for $i=1,2,\dotsc,p$ and $\sum_{i=1}^pl_i=k$.
   Set
   	\begin{align*}
   	&a_1=a_2=\dotsb=a_{l_1}=y_1,\\
   	&a_{l_1+1}=a_{l_1+2}=\dotsb=a_{ l_1+l_2}=y_2,\\
    &	\dotsb,\\
   	& a_{l_1+l_2+\dotsb+l_{p-1}+1}=a_{l_1+l_2+\dotsb+l_{p-1}+2}=\dotsb=a_{l_1+l_2+\dotsb+l_{p-1}+l_p}=y_p.
   	\end{align*}
Since $y_1,y_2,\dotsc,y_p$ are quasi-generic points,
there exists $N\in\bbn$ such that 
\[
\gamma(m_T(y_i,n),\mu_{y_i})<\frac{\eps}{2}
\]
for all $i=1,2,\dotsc,p$ and $n\geq N$.
By the assumption, there exists $x\in A$ and $n\in\bbn$ such that $n\geq N$, $T^{kn}x=x$ and   
\[
\min_{\sigma\in S_{kn}}\frac{1}{kn}
\sum_{i=0}^{kn-1}d(z_i,T^{\sigma(i)}x)<\frac{\eps}{2},
\] 
where 
$z_{i+(j-1)n}=T^ia_j $ for $ 0\leq i\leq n-1 $ and $ 1\leq j\leq k $.
Then applying Lemma~\ref{lem:finite-seq-metric}, we have 
\begin{align*}
\gamma\Bigl(\sum_{j=1}^p\lambda_j \mu_{y_j},\mu_x\Bigr)=
&\gamma\Bigl(\sum_{j=1}^p\frac{l_j}{k} \mu_{y_j},\mu_x\Bigr)\\
&\leq \gamma\Bigl(\sum_{j=1}^p\frac{l_j}{k} \mu_{y_j}, 
\sum_{j=1}^p\frac{l_j}{k}m_T(y_j,n)\Bigr)+
\gamma \Bigl(\sum_{j=1}^p\frac{l_j}{k}m_T(y_j,n), m_T(x,kn)\Bigr)\\
& \leq  \sum_{j=1}^p\frac{l_j}{k} \gamma\bigl( \mu_{y_j}, 
m_T(y_j,n)\bigr)+
\min_{\sigma\in S_{kn}}\frac{1}{kn}
\sum_{i=0}^{kn-1}d(z_i,T^{\sigma(i)}x)\\
&<\frac{\eps}{2}+\frac{\eps}{2}=\eps.
\end{align*}

Now we prove the necessity. 
Fix $k\geq 1$ and $y_1,y_2,\dotsc,y_k\in B$.  Clearly, $ \sum_{j=1}^{k}\frac{1}{k}\mu_{y_j} $ 
    is in $ \overline{\conv} (\{ \mu_y\colon y\in B \}) $. 
As $y_1,y_2,\dotsc,y_k$ are quasi-regular,
for every $\eps>0$, there exists $N_1\in\bbn$ such that 
for every $n>N_1$,
   	\[
   	\gamma\Big(\sum_{j=1}^{k}\frac{1}{k}\mu_{y_j} ,   
   	\sum_{j=1}^{k}\frac{1}{k} m_T(y_j,n)
   	 \Big)<\frac{\eps}{2}.
   	\]
By the assumption there exists $ x\in A$ with period $p$
   	such that 
   	\[
   	\gamma\Big(\sum_{j=1}^{k}\frac{1}{k}\mu_{y_j},   
   	\mu_x \Big)<\frac{\eps}{2}. 
   	\]
For every $N\in\bbn$, pick $n\in\bbn$ such that $n\geq \max\{N,N_1\}$ and $p|kn$.
Then  
\begin{align*}
\gamma\Bigl(\sum_{j=1}^{k}\frac{1}{k} m_T(y_j,n), m_T(x,kn)\Bigr)  
&\leq \gamma\Big(\sum_{j=1}^{k}\frac{1}{k} m_T(y_j,n),
   \sum_{j=1}^{k}\frac{1}{k}\mu_{y_j} \Bigr) +\gamma\Bigl(\sum_{j=1}^{k}\frac{1}{k}\mu_{y_j},\mu_x \Big)\\
&<\frac{\eps}{2}+\frac{\eps}{2} =\eps.
   \end{align*}
Applying Lemma~\ref{lem:finite-seq-metric}, we have
   \[
   \min_{\sigma\in S_{kn}}\frac{1}{kn}
   \sum_{i=0}^{kn-1}d(z_i,T^{\sigma(i)}x)<\eps,
   \]
   where 
   $ z_{i+(j-1)n}=T^iy_j $ for $ 0\leq i\leq n-1 $ and $ 1\leq j\leq k$.
\end{proof}

\begin{remark}
In Theorem~\ref{thm:per-dense-conv-B},
the points $y_1,y_2,\dotsc,y_k$ in $B$ usually are not pairwise distinct.
If in addition $A$ is a subset of $B$, then by Remark~\ref{rem:conv-hull} we can require that  $y_1,y_2,\dotsc,y_k$ have at most two distinct points.
\end{remark}

The concept of linkability is introduced in \cite{GK2018}.  
Let $ (X,T) $ be a Polish dynamical system 
and $ K\subset  \per(T)$.
We say that $ K $ is linkable if  for any $y_1,y_2\in K $, any $ \eps>0 $,  and any $ \lambda\in [0,1] $,  there exist   $p_1,p_2,q_1,q_2\in\bbn  $, and $ z\in K $ with $ T^{q_2}z=z$ such that
\begin{enumerate}
	\item  $ \lambda-\eps\leq \frac{p_1}{p_1+p_2}\leq \lambda+\eps$;
	\item $ p_1\leq q_1\leq (1+\eps)p_1  $ and 
	$ d(T^iz,T^iy_1)<\eps $ for $ 0\leq i\leq p_1-1 $; 
	\item $ p_2\leq q_2-q_1\leq (1+\eps)p_2 $ and 
	$ d(T^{i+q_1}z,T^iy_2)<\eps $ for $ 0\leq i\leq p_2-1 $.
\end{enumerate}
By \cite{GK2018}*{Lemma 4.13}, we can pick $p_1$ and $p_2$  such that they are both multiples of a given 
positive integer $r$. In particular, we can assume that $ T^{p_j}y_j=y_j $ for $ j=1,2$. 
It is shown in \cite{GK2018}*{Propositions 4.20, 4.22 and 4.23, respectively} that for a dynamical system 
with the periodic specification property ($\beta$-shift, $S$-gap shift, respectively) the set of periodic points is linkable. 

If we replace the metric used in the last two requirements in  linkability by the mean orbital pseudo metric, we  obtain the following  equivalence condition for the density of the convex hull.  

\begin{thm}\label{3.11}
Let $ (X,T) $ be a Polish dynamical system and 
$K \subset \per(T)$.
Then  $ \{\mu_x\colon x\in K\}$  is dense in  $ \overline{\conv}
\{\mu_x\colon x\in K\} $
if and only if 
 for any $y_1,y_2\in K $, any $ \eps>0 $,  and any $ \lambda\in [0,1] $,  there exist   $p_1,p_2,q_1,q_2\in\bbn  $, and $ z\in K $ 
 with $T^{p_1}y_1=y_1$, $T^{p_2}y_2=y_2$ and $ T^{q_2}z=z$ such that
\begin{enumerate}
	\item  $ \lambda-\eps\leq \frac{p_1}{p_1+p_2}\leq \lambda+\eps$, $ p_1\leq q_1\leq (1+\eps)p_1$, and $ p_2\leq q_2-q_1\leq (1+\eps)p_2 $;
	\item $ 
 \min\limits_{\sigma\in S_{q_2}}\displaystyle\frac{1}{q_2}
\sum_{i=0}^{q_2-1}d(x_i,T^{\sigma(i)}z)<\eps
$, 
where 
$x_i=T^iy_1$ for $ 0\leq i<q_1$ and $x_i=T^{i-q_1}y_2 $ 
for $q_1\leq i<q_2$.
\end{enumerate}
\end{thm}
\begin{proof}
As $d$ is a bounded metric on $X$, without loss of generality assume that the diameter
of $X$ is $1$.
We first prove the sufficiency.
By Remark~\ref{rem:conv-hull}, we only need to prove that 
$ \{\mu_x\colon x\in K\}$ is dense in 
$\{\lambda \mu_{y_1}+(1-\lambda)\mu_{y_2}\colon \lambda\in [0,1], y_1,y_2\in K\}$.
Fix $\eps>0$, $\lambda\in (0,1)$ and $y_1,y_2\in K$.
Then there exist $p_1,p_2,q_1,q_2\in\bbn  $, and $ z\in K $ 
 with $T^{p_1}y_1=y_1$, $T^{p_2}y_2=y_2$ and $ T^{q_2}z=z$ such that
\begin{enumerate}
	\item  $ \lambda-\eps\leq \frac{p_1}{p_1+p_2}\leq \lambda+\eps$, $ p_1\leq q_1\leq (1+\eps)p_1$, and $ p_2\leq q_2-q_1\leq (1+\eps)p_2 $;
	\item $ 
 \min\limits_{\sigma\in S_{q_2}}\displaystyle\frac{1}{q_2}
\sum_{i=0}^{q_2-1}d(x_i,T^{\sigma(i)}z)<\eps
$, 
where 
$x_i=T^iy_1$ for $ 0\leq i< q_1$ and $x_i=T^{i-q_1}y_2 $ 
for $q_1\leq i<q_2$.
\end{enumerate}
By the definition of $\gamma$ and Lemma~\ref{lem:finite-seq-metric}, one has
\begin{align*}
&\gamma\Bigl( \frac{p_1}{p_1+p_2} \mu_{y_1} + \frac{p_2}{p_1+p_2} \mu_{y_2},\mu_{z}\Bigr)=
\gamma\biggl( \frac{1}{p_1+p_2} \biggl(\sum_{i=0}^{p_1-1}\delta_{T^iy_1}+ \sum_{i=0}^{p_2-1}\delta_{T^iy_2}\biggr), \frac{1}{q_2} \sum_{i=0}^{q_2-1}\delta_{T^iz}\biggr)\\
&\qquad \qquad \leq \gamma\biggl( \frac{1}{p_1+p_2} \biggl(\sum_{i=0}^{p_1-1}\delta_{T^iy_1}+ \sum_{i=0}^{p_2-1}\delta_{T^iy_2}\biggr), \frac{1}{q_2}\biggl(\sum_{i=0}^{q_1-1}\delta_{T^iy_1}+ \sum_{i=0}^{q_2-q_1-1}\delta_{T^iy_2}\biggr)\biggr)\\
&\qquad \qquad \qquad \qquad  +\gamma\biggl(\frac{1}{q_2}\biggl(\sum_{i=0}^{q_1-1}\delta_{T^iy_1}+ \sum_{i=0}^{q_2-q_1-1}\delta_{T^iy_2}\biggr), \frac{1}{q_2} \sum_{i=0}^{q_2-1}\delta_{T^iz}\biggr)\\
&\qquad \qquad \leq \frac{q_2-(p_1+p_2)}{q_2}+\min_{\sigma\in S_{q_2}}\frac{1}{q_2}
\sum_{i=0}^{q_2-1}d(x_i,T^{\sigma(i)}z)\\
&\qquad \qquad< \eps+\eps = 2\eps.
\end{align*}
Then
\begin{align*}
\gamma(   \lambda \mu_{y_1}+(1-\lambda)\mu_{y_2},\mu_{z})
&\leq \gamma\Bigl(   \lambda \mu_{y_1}+(1-\lambda)\mu_{y_2}, 
\frac{p_1}{p_1+p_2} \mu_{y_1}+ \frac{p_2}{p_1+p_2} \mu_{y_2}\Bigr)\\
&\qquad \qquad +\gamma\Bigl( \frac{p_1}{p_1+p_2} \mu_{y_1}+ \frac{p_2}{p_1+p_2} \mu_{y_2},\mu_{z}\Bigr)\\
&\leq \Bigl| \lambda- \frac{p_1}{p_1+p_2}\Bigr|+\Bigl|1-\lambda - \frac{p_2}{p_1+p_2}\Bigr| +2\eps \\
&< \eps+\eps+2\eps = 4\eps.
\end{align*}

Now we prove the necessity. Fix $\eps>0$, $\lambda\in (0,1)$ and $y_1,y_2\in K$. Then there exist  $p_1',p_2'\in\bbn$ with $T^{p_1'}y_1=y_1$, $T^{p_2'}y_2=y_2$ such that 
$\lambda-\eps\leq \frac{p_1'}{p_1'+p_2'}\leq \lambda+\eps$.
By the assumption, there exists $z\in K$ such that 
\[  
\gamma\Bigl( \frac{p_1'}{p_1'+p_2'} \mu_{y_1} + \frac{p_2'}{p_1'+p_2'} \mu_{y_2},\mu_{z}\Bigr)<\eps.
\] 
Let $q_2$ be a multiple of $(p_1'+p_2')$ with $T^{q_2}z=z$.
Let $p_1=\frac{p_1'q_2}{p_1'+p_2'}  $,
$p_2=\frac{p_2'q_2}{p_1'+p_2'}  $,  and $q_1=p_1$.
Applying  Lemma~\ref{lem:finite-seq-metric}, one has
\begin{align*}
  \gamma\Bigl( \frac{p_1'}{p_1'+p_2'} \mu_{y_1} + \frac{p_2'}{p_1'+p_2'} \mu_{y_2},\mu_{z}\Bigr)
 & = \gamma \biggl( \frac{1}{q_2} \biggl(\sum_{i=0}^{p_1-1}
  \delta_{T^{i}{y_1}} + 
  \sum_{i=0}^{p_2-1} \delta_{T^{i}{y_2}}  \biggr), \frac{1}{q_2}\sum_{i=0}^{q_2-1}\delta_{T^{i}{z}}  \biggr)\\
  &=  \min\limits_{\sigma\in S_{q_2}}\displaystyle\frac{1}{q_2}
\sum_{i=0}^{q_2-1}d(x_i,T^{\sigma(i)}z)<\eps,
\end{align*}
where 
$x_i=T^iy_1$ for $ 0\leq i< q_1$ and $x_i=T^{i-q_1}y_2 $ 
for $q_1\leq i<q_2$.
\end{proof}

The following result shows that linkability is stronger than the equivalent condition in Theorem ~\ref{3.11}.

\begin{prop}\label{1}
Let $ (X,T) $ be a Polish dynamical system.
If $K\subset \per(T)$ is linkable,
then  for any $y_1,y_2\in K $, any $ \eps>0 $,  and any $ \lambda\in [0,1] $,  there exist   $p_1,p_2,q_1,q_2\in\bbn  $, and $ z\in K $ 
 with $T^{p_1}y_1=y_1$, $T^{p_2}y_2=y_2$ and $ T^{q_2}z=z$ such that
\begin{enumerate}
	\item  $ \lambda-\eps\leq \frac{p_1}{p_1+p_2}\leq \lambda+\eps$, $ p_1\leq q_1\leq (1+\eps)p_1$, and $ p_2\leq q_2-q_1\leq (1+\eps)p_2 $;
	\item $ 
 \displaystyle\frac{1}{q_2}
\sum_{i=0}^{q_2-1}d(x_i,T^{i}z)<\eps
$, 
where 
$x_i=T^iy_1$ for $ 0\leq i<q_1$ and $x_i=T^{i-q_1}y_2 $ 
for $q_1\leq i<q_2$.
\end{enumerate}
\end{prop}

\begin{proof}
As $d$ is a bounded metric on $X$, without loss of generality assume that the diameter
of $X$ is $1$.
Fix  $y_1,y_2\in K $, $ \eps>0 $,  and  $ \lambda\in [0,1] $.
There exist $p_1,p_2,q_1,q_2\in\bbn  $, and $ z\in K $ with $ T^{q_2}z=z$ such that
\begin{enumerate}
	\item  $ \lambda-\eps\leq \frac{p_1}{p_1+p_2}\leq \lambda+\eps$;
	\item $ p_1\leq q_1\leq (1+\eps)p_1  $ and 
	$ d(T^iz,T^iy_1)<\eps $ for $ 0\leq i\leq p_1-1 $; 
	\item $ p_2\leq q_2-q_1\leq (1+\eps)p_2 $ and 
	$ d(T^{i+q_1}z,T^iy_2)<\eps $ for $ 0\leq i\leq p_2-1 $.
\end{enumerate}
Let $x_i=T^iy_1$ for $ 0\leq i<q_1$ and $x_i=T^{i-q_1}y_2 $ 
for $q_1\leq i<q_2$.
Then 
\begin{align*}
\frac{1}{q_2}\sum_{i=0}^{q_2-1} d(x_i,T^iz) &= 
 \frac{1}{q_2}  \biggl( \sum_{i=0}^{p_1-1} d(T^iy_1,T^iz) +    
   \sum_{i=p_1}^{q_1-1} d(T^iy_1,T^iz) \\
    &\qquad\qquad  +\sum_{i=q_1}^{p_2+q_1-1} d(T^{i-q_1}y_2 ,T^{i}z)  
    + \sum_{i=p_2+q_1}^{q_2-1} d(T^{i-q_1}y_2 ,T^{i}z)  
    \biggr)\\
    & \leq \frac{1}{q_2} ( p_1\eps + q_1-p_1 + p_2\eps + q_2-p_2-q_1)\\
    &<2\eps. \qedhere 
\end{align*}
\end{proof}

\begin{cor}[\cite{GK2018}*{Theorem 4.25}]
Let $ (X,T) $ be a Polish dynamical system.
If $K\subset \per(T)$ is linkable, then $ \{\mu_x\colon x\in K \} $ is dense in $ \overline{\conv}\{\mu_x\colon x\in K \} $.
\end{cor}

\begin{cor}
Let $ (X,T) $ be a Polish dynamical system.
If $K\subset \per(T)$ is linkable and $\{\mu_x\colon x\in K\}$ is dense in $\calm_T^e(X)$, then $\{\mu_x\colon x\in K\}$ is dense in $\calm_T(X)$.
\end{cor}

Let $(X,T)$ be a dynamical system and  $n\geq 1$. We say that a collection $\{D_0,D_1,\dotsc,D_{k-1}\}$
of closed subsets of $X$ is a periodic decomposition if 
$T(D_i)\subset D_{i+1 \pmod k}$ for each $0\leq i<k$ and $\bigcup_{i=0}^{k-1} D_i=X$.
 
\begin{prop}\label{prop:per-decomp}
Let $(X,T)$ be a Polish dynamical system with $T$ being uniformly continuous.
Assume that $\{D_0,D_1,\dotsc,D_{k-1}\}$ is a periodic decomposition.
\begin{enumerate}
    \item If $\{\mu_x\colon x\in \per(T^k)\cap D_0\}$ is dense in $\calm_{T^k}^e(X)$, then $\{\mu_x\colon x\in \per(T)\}$ is dense in $\calm_T^e(X)$.
    \item If $\{\mu_x\colon x\in \per(T^k)\cap D_0\}$ is dense in $\calm_{T^k}(X)$, then $\{\mu_x\colon x\in \per(T)\}$ is dense in $\calm_T(X)$.
\end{enumerate}
\end{prop}
\begin{proof}
Let $\mu\in\calm_T(X)$. As $T^{-1}(D_i)\supset D_{i-1\pmod k}$, $\mu(D_0)=\mu(D_1)=\dotsb=\mu(D_{k-1})>0$.
For each $i=0,1,\dotsc,k-1$, let $\mu_i=\mu|_{D_i}$, i.e.
$\mu_i(A)=\frac{\mu(A\cap D_i)}{\mu(D_i)}$ for every Borel subset $A$ of $X$.
Then each $\mu_i$ is $T^k$-invariant and $\supp(\mu_i)\subset D_i$.

(1) If $\mu$ is ergodic, then it is easy to check that each $\mu_i$ is ergodic for $T^k$. Pick a generic point $y$ for $\mu_0$ with respect to $T^k$. Then for each $i=1,2,\dotsc,k-1$,
$T^iy$ is generic for $\mu_i$  with respect to $T^k$ and $y$ is also generic for $\mu$ with respect to $T$.
Fix $\eps>0$ and $N\in\bbn$.
By the uniform continuity of $T$, there exists $\delta>0$ such that 
for every $a,b\in X$ with $d(a,b)<\delta$ one has $d(T^ia,T^ib)<\eps$ for $i=0,1,\dotsc,k-1$.
As $\{\mu_x\colon x\in \per(T^k)\cap D_0\}$ is dense in $\calm_{T^k}^e(X)$, 
by Theorem~\ref{thm:per-dense-B} 
there exist $x\in \per(T^k)\cap D_0$
and $n\in\bbn$ such that $n\geq N$, $T^n x= x$ and 
\[
\min_{\sigma\in S_{n} }\frac{1}{n}\sum_{i=0}^{n-1}d((T^k)^ix,(T^k)^{\sigma(i)}y )<\delta.
\]
Then $x\in \per(T)$, $T^{kn}x=x$ and 
\[
\min_{\sigma\in S_{kn} }\frac{1}{kn}\sum_{i=0}^{kn-1}d(T^ix,T^{\sigma(i)}y )<\eps.
\]
By Theorem~\ref{thm:per-dense-B} again, $\{\mu_x\colon x\in \per(T)\}$ is dense in $\calm_T^e(X)$.

Using Theorem~\ref{thm:per-dense-conv-B} instead of Theorem~\ref{thm:per-dense-B},
similar to (1) one can prove  (2).
\end{proof}

\begin{cor}
Let $f\colon  [0, 1] \to [0, 1]$ be a continuous map.
If $f$ is transitive then  $\{\mu_x\colon x\in \per(f)\}$ is dense in $\calm_f([0,1])$.
\end{cor}
\begin{proof}
If $f$ is weakly mixing, then by \cite{Rue2017}*{Theorem 3.4} 
$f$ has the periodic specification property. Hence $\{\mu_x\colon x\in \per(f)\}$ is dense in $\calm_f([0,1])$. 
If $f$ is transitive but not weakly mixing, then by \cite{Rue2017}*{Theorem 2.19} there exists $c\in (0,1)$
such that $\{[0,c],[c,1]\}$ is a periodic decomposition.
Moreover $([0,c],f^2)$ is weakly mixing.
Now the result follows from Proposition~\ref{prop:per-decomp}.
\end{proof}

\section{Dynamical systems with asymptotic orbital average
shadowing property}
Let $(X,T)$ be a Polish dynamical system  and  $d$ be   a compatible bounded complete metric  on $X$.
A sequence $\{x_i\}_{i=0}^\infty$ in $X$ is called an asymptotic average pseudo-orbit for $T$ if 
\[
\lim_{n\to\infty}\frac{1}{n}\sum_{i=0}^{n-1}d(Tx_i,x_{i+1})=0.
\]
Following \cite{Gu2007}, we say that a Polish dynamical system $(X,T)$  has the asymptotic average shadowing property if for every asymptotic average pseudo-orbit  $\{x_i\}_{i=0}^\infty$ of $X$ there exists $x\in X$ such that
\[
\lim_{n\to\infty}\frac{1}{n}\sum_{i=0}^{n-1}d(T^ix,x_{i})=0.
\]
Given a sequence $\{x_i\}_{i=0}^\infty$ in $X$ and $n\in\bbn$, define 
\[
m(\{x_i\}_{i=0}^\infty,n)=\frac{1}{n}\sum_{i=0}^{n-1}\delta_{x_i}.
\]
We denote by  $\calv(\{x_i\}_{i=0}^\infty)$ the collection of all the limit points of subsequences of $\{m(\{x_i\}_{i=0}^\infty,n)\}_{n=1}^\infty$.
It is clear that $\calv(\{x_i\}_{i=0}^\infty)$ is a  closed subset of $\calmx$, 
but it may be empty.
As  $\lim_{n\to\infty} \gamma(m(\{x_i\}_{i=0}^\infty,n+1),m(\{x_i\}_{i=0}^\infty,n))=0$, the set   $\calv(\{x_i\}_{i=0}^\infty)$ is connected.
If $X$ is a compact metric space, then $(\calmx,\gamma)$ is also a compact metric space  and $\calv(\{x_i\}_{i=0}^\infty)$ is  not empty.
For $x\in X$,  denote  $\calv(\{T^ix\}_{i=0}^\infty)$  by $\calv_T(x)$ for shortly.

\begin{lem} \label{lem:lim-0-V}
Let $X$ be a Polish space and $\{x_i\}_{i=0}^\infty$, $\{y_i\}_{i=0}^\infty $ be two sequences in $X$ provided that $\calv(\{x_i\}_{i=0}^\infty)\neq\emptyset$. 
If \[
\lim_{n\to\infty} \min_{\sigma\in S_n }\frac{1}{n}\sum_{i=0}^{n-1}d(x_i,y_{\sigma(i)})=0.
\] 
then $\calv(\{x_i\}_{i=0}^\infty)=\calv(\{y_i\}_{i=0}^\infty)$.
\end{lem}
\begin{proof}
It is sufficient to prove that $\calv(\{x_i\}_{i=0}^\infty)\subset \calv(\{y_i\}_{i=0}^\infty)$.
Fix $\mu\in  \calv(\{x_i\}_{i=0}^\infty)$.
There exists an increasing sequence $\{n_k\}$ such that 
\[
\lim_{k\to\infty} m(\{x_i\}_{i=0}^\infty,n_k) = \mu.
\]
For every $\eps>0$ there exists $N_1\in\bbn$ such that for any 
$n\geq N_1$, 
\[
\gamma(m(\{x_i\}_{i=0}^\infty,n_k),\mu)<\frac{\eps}{2}.
\]
By the assumption, there exists $N_2\in\bbn$ such that for any 
$n\geq N_2$,  
\[
\min_{\sigma\in S_n }\frac{1}{n}\sum_{i=0}^{n-1}d(x_n,y_{\sigma(n)})<\frac{\eps}{2}.
\]
By Lemma~\ref{lem:finite-seq-metric}, for any 
$n\geq N_2$,  
\[
\gamma(m(\{x_i\}_{i=0}^\infty,n), m(\{y_i\}_{i=0}^\infty,n)<\frac{\eps}{2}.
\]
Let $N=\max\{N_1,N_2\}$. Then for any $n\geq N$, 
\begin{align*}
\gamma(m(\{y_i\}_{i=0}^\infty,n_k),\mu)&\leq \gamma(m(\{y_i\}_{i=0}^\infty,n_k), m(\{x_i\}_{i=0}^\infty,n_k))+\gamma(m(\{x_i\}_{i=0}^\infty,n_k),\mu)\\
&< \frac{\eps}{2}+\frac{\eps}{2}=\eps.
\end{align*}
This shows that 
\[
\lim_{k\to\infty} m(\{y_i\}_{i=0}^\infty,n_k) = \mu,
\]
and then $\mu \in \calv(\{y_i\}_{i=0}^\infty)$.
\end{proof}

\begin{remark}
It should be noticed that the converse of Lemma \ref{lem:lim-0-V} is not true. The section 3.1 of \cite{CKLP2022} demonstrates an example that 
there exist two points $x,y$ in the full shift $(\{0,1\}^\infty,T)$ that $\calv_T(x)=\calv_T(y)$ and 
\[
\limsup_{n\to\infty} \min_{\sigma\in S_n }\frac{1}{n}\sum_{i=0}^{n-1}d(T^ix,T^{\sigma(i)}y)>0.
\] 
The following result shows that 
the converse of Lemma \ref{lem:lim-0-V} can be true under some additional conditions.
\end{remark}

\begin{lem}\label{lem:lim-0-V-eq}
Let $X$ be a compact metric space and $\{x_i\}_{i=0}^\infty$, $\{y_i\}_{i=0}^\infty $ be two sequences in $X$ provided that $\calv(\{x_i\}_{i=0}^\infty)$ is a singleton.
Then   $\calv(\{x_i\}_{i=0}^\infty)=\calv(\{y_i\}_{i=0}^\infty)$ 
if and only if 
\[
\lim_{n\to\infty} \min_{\sigma\in S_n }\frac{1}{n}\sum_{i=0}^{n-1}d(x_i,y_{\sigma(i)})=0.
\] 
\end{lem}
\begin{proof}
The sufficiency  follows from Lemma~\ref{lem:lim-0-V}.
We only need to prove the necessity.
Assume that $\calv(\{x_i\}_{i=0}^\infty) = \{\mu \}$.
As $X$ is a compact metric space, $(\calmx,\gamma)$ is also a compact metric space. Then 
\[
\lim_{n\to\infty} m(\{x_i\}_{i=0}^\infty,n) = \mu=\lim_{n\to\infty} m(\{y_i\}_{i=0}^\infty,n).
\]
For every $\eps>0$ there exists $N\in\bbn$ such that for any 
$n\geq N$, 
\[
\gamma(m(\{x_i\}_{i=0}^\infty,n),\mu)<\frac{\eps}{2},
\]
and 
\[
\gamma(m(\{y_i\}_{i=0}^\infty,n),\mu)<\frac{\eps}{2}.
\]
By Lemma~\ref{lem:finite-seq-metric}, for any 
$n\geq N$,  
\begin{align*}
\min_{\sigma\in S_n }\frac{1}{n}\sum_{i=0}^{n-1}d(x_i,y_{\sigma(i)})&=
\gamma(m(\{x_i\}_{i=0}^\infty,n), m(\{y_i\}_{i=0}^\infty,n))\\
&< \gamma(m(\{x_i\}_{i=0}^\infty,n),\mu) + \gamma(m(\{y_i\}_{i=0}^\infty,n),\mu)\\
&< \frac{\eps}{2}+\frac{\eps}{2}=\eps.
\end{align*}
Then 
\[
\lim_{n\to\infty} \min_{\sigma\in S_n }\frac{1}{n}\sum_{i=0}^{n-1}d(x_i,y_{\sigma(i)})=0. \qedhere 
\] 
\end{proof}

The following result is essentially contained in \cite{Sig1977}*{Remark 1}
for compact dynamical system. Here we provide a proof for complement.

\begin{prop} \label{prop:V-AAPO}
Let $(X,T)$ be a Polish dynamical system.
Then for every non-empty compact connected  subset $V$ of $\calm_T(X)$,
there exists an asymptotic average pseudo-orbit $\{x_i\}_{i=0}^\infty$ for  $T$ such that $\calv(\{x_i\}_{i=0}^\infty )=V$.
\end{prop} 
\begin{proof}
As $d$ is a bounded metric on $X$, without loss of generality assume that the diameter
of $X$ is $1$.
Since $V$ is a non-empty compact connected subset of $\calm_T(X)$,
it is easy to find a dense sequence $\{\mu_n\}_{n=1}^\infty$ in $V$ 
such that
$\gamma(\mu_n,\mu_{n+1})\to 0$ as $n\to\infty$.
As $V$ has no isolated points, $V$ is the collection of limits of convergent subsequences of $\{\mu_n\}_{n=1}^\infty$.
Let $\eps_n= \gamma(\mu_n,\mu_{n+1})$ for all $n\in\bbn$.
For each $n\in\bbn$, by the ergodic decomposition,  
there exists a  large enough positive integer $p_n$ and ergodic measures $\mu^n_1,\mu^n_2,\dotsc,\mu^n_{p_n}$ such that 
\[  
\gamma\biggl(\mu_n, \frac{1}{p_n}\sum_{j=1}^{p_n} \mu^n_j \biggr)<\eps_n.
\]
For each $j=1,2,\dotsc,p_n$, as $\mu^n_j$ is ergodic, pick a generic point $y^n_j$ for $\mu^n_j$.
Then there exists a  large enough positive integer $q_n$ such that 
\[
\gamma\bigl(\mu^n_j, m_T(y^n_j,q_n)\bigr)<\eps_n
\]
for $j=1,2,\dotsc,p_n$. We require that $p_{n+1}\geq 2^n p_nq_n$ and $q_{n+1}\geq 2p_{n+1}$ for all $n\in\bbn$.

Let $Q_0=0$, $P_n=p_n q_n p_{n+1}q_{n+1}$ and $Q_n=\sum_{i=1}^n P_i$ for all $n\in\bbn$.
For each $i\in\bbn$, there exists a unique $n\in\bbn$ such that 
$i=Q_{n-1} + k p_nq_n+ j q_n +r$ with $k\in [0,p_{n+1}q_{n+1})$, 
$j\in [0,p_n-1)$ and $r\in [0,q_n)$ and define  
$x_i=T^r y^n_{j+1}$.

For each $N\in\bbn$, there exists a unique $n\in\bbn$ such that $N=Q_{n-1} + k p_nq_n+ j q_n +r$ with $k\in [0,p_{n+1}q_{n+1})$, 
$j\in [0,p_n-1)$, and $r\in [0,q_n)$.
Then 
\begin{align*}
\frac{1}{N} \sum_{i=0}^{N-1}d(Tx_i,x_{i+1}) & \leq 
\frac{1}{N} \biggl(\sum_{i=1}^{n-1} p_i p_{i+1}q_{i+1}+(k+1) p_n\biggr)  \\
&\leq \frac{1}{Q_{n-1} }\sum_{i=1}^{n-1} p_i p_{i+1}q_{i+1} +\frac{(k+1) p_n}{Q_{n-1}+kp_nq_n} \to 0
\end{align*}
as $N\to\infty$.
This shows that 
$\{x_i\}_{i=0}^\infty$ is an asymptotic average pseudo-orbit.

By the construction of $\{x_i\}$, for every $n\in\bbn$ and $k\in [0,p_{n+1}q_{n+1})$, one has 
\[
\frac{1}{p_n} \sum_{j=1}^{p_n} m_T(y^n_j,q_n)
= \frac{1}{k p_n q_n} \sum_{i=Q_{n-1}}^{Q_{n-1}+ k p_n q_n-1}  \delta_{x_i}.
\]
Then by the definition of the metric $\gamma$,  for each $n\in\bbn$, 
\begin{align*}
\gamma(\mu_n,m(\{x_i\}_{i=0}^\infty,Q_n))
&\leq \gamma\biggl(\mu_n, \frac{1}{p_n} \sum_{j=1}^{p_n} m_T(y^n_j,q_n)\biggr) +
\gamma \biggl(
\frac{1}{p_n} \sum_{j=1}^{p_n} m_T(y^n_j,q_n),
m(\{x_i\}_{i=0}^\infty,Q_n)
\biggr)\\
&<\gamma\biggl(\mu_n, \frac{1}{p_n}\sum_{j=1}^{p_n} \mu^n_j \biggr)+ \gamma\biggl(\frac{1}{p_n}\sum_{j=1}^{p_n} \mu^n_j, \frac{1}{p_n} \sum_{j=1}^{p_n} m_T(y^n_j,q_n)\biggr)\\
&\qquad \qquad \qquad  +\gamma \biggl(
\frac{1}{p_n} \sum_{j=1}^{p_n} m_T(y^n_j,q_n),
m(\{x_i\}_{i=0}^\infty,Q_n)
\biggr)\\
&< 2\eps_n + \frac{2Q_{n-1}}{Q_n} \to 0
\end{align*}
as $n\to\infty$.
This shows that $V\subset \calv(\{x_i\}_{i=0}^\infty)$.
 
Fix $\nu \in \calv(\{x_i\}_{i=0}^\infty)$.
There exists an increasing sequence $\{t_\ell \}$ such that 
\[ 
\lim_{\ell\to\infty}m(\{x_i\}_{i=0}^\infty,t_\ell)=\nu.
\]
For each $t_\ell$, there exists a unique $n_\ell\in\bbn$ such that 
\[
t_\ell=Q_{n_\ell-1} + k_\ell p_{n_\ell} q_{n_\ell}+ j_\ell q_{n_\ell} +r_\ell
\]
with $k_\ell \in [0,p_{n_\ell+1}q_{n_\ell+1})$, 
$j_\ell \in [0,p_{n_\ell}-1)$ and $r_\ell \in [0,q_{n_\ell})$.
As $\{\mu_n\}_{n=1}^\infty$ is a sequence in a compact set $V$,
without loss of generality, assume that the subsequence 
$\{\mu_{n_\ell-1}\}$ converges to $\nu'$.
Let $\alpha_\ell = \frac{t_\ell - Q_{n_\ell-1}}{t_\ell}$. Then
\begin{align*}
&\alpha_\ell 
\gamma\biggl(\frac{1}{t_\ell - Q_{n_\ell-1}}\sum_{i=Q_{n_\ell-1}}^{t_\ell-1}\delta_{x_i},\mu_{n_\ell}\biggr)\\
&\quad \leq 
\alpha_\ell \gamma\biggl(
\frac{1}{t_\ell - Q_{n_\ell-1}}\sum_{i= Q_{n_\ell-1}}^{t_\ell-1}\delta_{x_i}, 
\frac{1}{p_{n_\ell}} \sum_{j=1}^{p_{n_\ell}} m_T(y^{n_\ell}_j,q_{n_\ell})\biggl) + 
\alpha_\ell 
\gamma\biggl(\frac{1}{p_{n_\ell}} \sum_{j=1}^{p_{n_\ell}} m_T(y^{n_\ell}_j,q_{n_\ell}), \mu_{n_\ell}\biggr)\\ 
&\quad \leq \frac{2p_{n_\ell} q_{n_\ell}+1}{Q_{n_\ell-1}}+ 2\eps_{n_\ell}
\to 0
\end{align*}
as $\ell\to\infty$ 
and 
\begin{align*}
    &\gamma(m(\{x_i\}_{i=0}^\infty,t_\ell),
    (1-\alpha_\ell) \mu_{n_\ell-1}+\alpha_\ell\mu_{n_\ell})\\
    &\quad  
    =  \gamma\biggl((1-\alpha_\ell )m(\{x_i\}_{i=0}^\infty,Q_{n_\ell-1})+
    \alpha_\ell \frac{1}{t_\ell - Q_{n_\ell-1}}
    \sum_{i= Q_{n_\ell-1}}^{t_\ell-1}\delta_{x_i},
    (1-\alpha_\ell) \mu_{n_\ell-1}+\alpha_\ell\mu_{n_\ell}\biggr)\\
    &\quad \leq 
    (1-\alpha)\gamma(m(\{x_i\}_{i=0}^\infty,Q_{n_\ell-1}),  \mu_{n_\ell-1})
    +\alpha_\ell \gamma\biggl(\frac{1}{t_\ell - Q_{n_\ell-1}}\sum_{i=Q_{n_\ell-1}}^{t_\ell-1}\delta_{x_i},\mu_{n_\ell}\biggr)
    \to 0
\end{align*}
as $\ell\to\infty$. 
On the other hand, $\gamma((1-\alpha_\ell) \mu_{n_\ell-1}+\alpha_\ell\mu_{n_\ell},\mu_{n_\ell-1}) \leq \gamma(\mu_{n_\ell},\mu_{n_\ell-1})\to 0$ as $\ell\to \infty$.
This shows that $\nu=\nu'$ and then $ \calv(\{x_i\}_{i=0}^\infty)\subset V$.
\end{proof}

We say that a Polish dynamical system $(X,T)$  has the asymptotic orbital average shadowing property if  for every asymptotic average pseudo-orbit  $\{x_i\}_{i=0}^\infty$ of $X$,  there exists $x\in X$ such that
\[
\lim_{n\to\infty} \min_{\sigma\in S_n }\frac{1}{n}\sum_{i=0}^{n-1}d(x_i,T^{\sigma(i)}x)=0.
\]

Now we are ready to prove Theorem~\ref{thm:AOASP-V}.

\begin{proof}[Proof of Theorem~\ref{thm:AOASP-V}]
Let $V$ be a non-empty compact connected subset of $\calm_T(X)$.
By Proposition \ref{prop:V-AAPO}, there exists an asymptotic average pseudo-orbit $\{x_i\}_{i=0}^\infty$ for  $T$ such that $\calv(\{x_i\}_{i=0}^\infty )=V$.
Since $(X,T)$ has the asymptotic orbital average shadowing property,   there exists some $x\in X$ such that 
\[
\lim_{n\to \infty}\min_{\sigma\in S_n}\frac{1}{n}\sum_{i=0}^{n-1}d(x_i,T^{\sigma(i)}x)=0.
\]
By Lemma \ref{lem:lim-0-V}, $\calv_T(x)= \calv(\{x_i\}_{i=0}^\infty )$.
Then $\calv_T(x)=V$. 
\end{proof}

\begin{cor}[\cite{DTY2015}*{Theorem 1.4} and \cite{KLO2017}*{Theorem 21}]
If a  compact dynamical system $(X,T)$ has the asymptotic average shadowing property, then for every non-empty closed connected subset $V$ of $\calm_T(X)$, there exists $x\in X$ with $\calv_T(x)=V$.
\end{cor}

\begin{prop} \label{prop:per-decomp-AOASP}
Let $(X,T)$ be a Polish dynamical system with $T$ being uniformly continuous.
Assume that $\{D_0,D_1,\dotsc,D_{k-1}\}$ is a periodic decomposition. 
If $(D_0,T^k)$ has the asymptotic orbital average
shadowing property then for every non-empty compact connected  subset $V$ of $\calm_T(X)$, there exists $x\in X$ with $\calv_T(x)=V$.  
\end{prop}
\begin{proof}
Let $V$ be a non-empty compact connected subset of $\calm_T(X)$.
By Proposition \ref{prop:V-AAPO}, there exists an asymptotic average pseudo-orbit $\{x_i\}_{i=0}^\infty$ for  $T$ such that $\calv(\{x_i\}_{i=0}^\infty )=V$. 
Moreover in the construction of $\{x_i\}_{i=0}^\infty$ in the proof of   Proposition \ref{prop:V-AAPO}, we can require that each $y_j^n\in D_0$ and $q_n$ is a multiple of $k$. 
Then $\{x_{ki}\}_{i=0}^\infty$ is an asymptotic average pseudo-orbit for $(D_0,T^k)$.  
Since $(D_0,T^k)$ has the asymptotic orbital average shadowing property,   there exists some $x\in D_0$ such that 
\[
\lim_{n\to \infty}\min_{\sigma\in S_n}\frac{1}{n}\sum_{i=0}^{n-1}d(x_{kn},(T^k)^{\sigma(i)}x)=0.
\]
By the construction of $\{x_i\}_{i=0}^\infty$ and the uniform continuity of $T$, we have 
\[
\lim_{n\to \infty}\min_{\sigma\in S_n}\frac{1}{n}\sum_{i=0}^{n-1}d(x_{n},T^{\sigma(i)}x)=0.
\]
By Lemma \ref{lem:lim-0-V}, $\calv_T(x)= \calv(\{x_i\}_{i=0}^\infty )$.
Then $\calv_T(x)=V$. 
\end{proof}

\begin{cor}
Let $f\colon [0, 1] \to [0, 1]$ be a continuous map. If $f$ is transitive then  for every non-empty closed connected subset $V$ of $\calm_f([0,1])$, there exists $x\in X$ with $\calv_f(x)=V$.
\end{cor}
\begin{proof}
If $f$ is weakly mixing, then by \cite{Rue2017}*{Theorem 3.4} 
$f$ has the periodic specification property. 
By \cite{KO2010}*{Theorem 3.8}, $f$ has the asymptotic average shadowing property.
Hence for every non-empty closed connected subset $V$ of $\calm_f([0,1])$, there exists $x\in X$ with $\calv_f(x)=V$.
If $f$ is transitive but not weakly mixing, then by \cite{Rue2017}*{Theorem 2.19} there exists $c\in (0,1)$
such that $\{[0,c],[c,1]\}$ is a periodic decomposition.
Moreover $([0,c],f^2)$ is weakly mixing.
Now the result follows from Proposition~\ref{prop:per-decomp-AOASP}.
\end{proof}

We also have the following equivalent condition for when every non-empty compact connected subset of the space of invariant measures has a generic point for compact dynamical systems, but it seems not easy to be verified.

\begin{prop}
Let $(X,T)$ be a compact dynamical system. Then the following two statements are equivalent:
\begin{enumerate}
    \item for every non-empty compact connected  subset $V$ of $\calm_T(X)$, there exists $x\in X$ with $\calv_T(x)=V$;
    \item for every asymptotic average pseudo-orbit  $\{x_i\}_{i=0}^\infty$ of $X$, 
    there exists a asymptotic average pseudo-orbit  $\{y_i\}_{i=0}^\infty$ of $X$ and a point $x\in X$  such that  $\calv(\{x_i\}_{i=0}^\infty ) = \calv(\{y_i\}_{i=0}^\infty )$ and
\[
\lim_{n\to\infty} \min_{\sigma\in S_n }\frac{1}{n}\sum_{i=0}^{n-1}d(y_i,T^{\sigma(i)}x)=0.
\]
\end{enumerate}
\end{prop}
\begin{proof}
(1) $\Rightarrow$ (2)  
Assume that 
 $\{x_i\}_{i=0}^\infty$ is an asymptotic average pseudo-orbit in $X$.
Then $ \calv(\{x_i\}_{i=0}^\infty )$ is a non-empty compact connected subset of $\calm_T(X)$.
By the assumption, there exists $x\in X$ with $\calm_T(X) = \calv(\{x_i\}_{i=0}^\infty )$. 
Then the orbit $\{T^i x\}_{i=0}^\infty$ and the point $x$ are as required.

(2) $\Rightarrow$ (1) is similar to the proof of Theorem \ref{thm:AOASP-V}.
\end{proof} 

\begin{prop}
Let $(X,T)$ be a compact dynamical system. Then the following two statements are equivalent:
\begin{enumerate}
    \item every invariant measure has a generic point;
    \item for every asymptotic average pseudo-orbit  $\{x_i\}_{i=0}^\infty$ of $X$ with $\calv(\{x_i\}_{i=0}^\infty )$ a singleton, there exists a point $x\in X$ such that
\[
\lim_{n\to\infty} \min_{\sigma\in S_n }\frac{1}{n}\sum_{i=0}^{n-1}d(x_i,T^{\sigma(i)}x)=0.
\]
\end{enumerate}
\end{prop}
\begin{proof}
(1) $\Rightarrow$ (2)  Assume that 
 $\{x_i\}_{i=0}^\infty$ is an asymptotic average pseudo-orbit in $X$ with $\calv(\{x_i\}_{i=0}^\infty )$ a singleton.
Denote by $\mu$ the unique element in $ \calv(\{x_i\}_{i=0}^\infty )$.
Then $\mu$ is an invariant measure.
Let $x$ be a generic point for $\mu$. Then $\calv_T(x)=\{\mu\}$.
According to Lemma \ref{lem:lim-0-V-eq}, one has
\[
\lim_{n\to\infty} \min_{\sigma\in S_n }\frac{1}{n}\sum_{i=0}^{n-1}d(x_i,T^{\sigma(i)}x)=0.
\]

(2) $\Rightarrow$ (1) is similar to the proof of Theorem \ref{thm:AOASP-V}.
\end{proof}

\begin{remark}
Let $ (X,T) $ be a Polish dynamical system.
It is shown in \cite{GK2018}*{Corollary 5.3} that
if $K\subset \per(T)$ is linkable and $\{\mu_x\colon x\in K\}$ is dense in $\calm_T^e(X)$ then for every non-empty compact connected  subset $V$ of $\calm_T(X)$ there exists $x\in X$ with $\calv_T(x)=V$. 
It will be interesting to know whether $(X,T)$ has the asymptotic average shadowing property in this case.
\end{remark}

\noindent \textbf{Acknowledgments:} J. Li was supported in part by NNSF of China (Grant Nos. 12222110, 12171298).
The authors would like to thank Prof. Weisheng Wu for helpful suggestions.

\end{document}